\documentclass[12pt,reqno]{amsart}

\usepackage[left=3.5cm,right=3.5cm]{geometry}
\usepackage{amssymb}
\usepackage{graphicx}
\usepackage{amscd}
\usepackage[pagebackref]{hyperref}
\usepackage{color}
\usepackage{tabularx}
\usepackage[table]{xcolor}
\usepackage{float}
\usepackage{graphics,amsmath,amssymb}
\usepackage{amsthm}
\usepackage{amsfonts}
\usepackage{latexsym}
\usepackage{epsf}
\usepackage{xifthen}
\usepackage{mathrsfs}
\usepackage{dsfont}
\usepackage{makecell}
\usepackage{subfig}
\usepackage{amsmath}
\allowdisplaybreaks[4]
\usepackage{listings}
\usepackage{etoolbox}
\usepackage{fancyhdr}
\usepackage{pdflscape}
\usepackage[title,toc,titletoc]{appendix}
\usepackage{enumitem}
\usepackage[noadjust]{cite}
\usepackage{tikz}
\usetikzlibrary{automata,positioning,arrows}
\usepackage{young}
\usepackage[object=vectorian]{pgfornament} 
\usepackage{lipsum,tikz}
\usepackage{multirow}
\usepackage[OT2,T1]{fontenc}
\usepackage{mathtools}

\usepackage[normalem]{ulem} 
\usepackage{soul}

\hypersetup{
	colorlinks=true, 
	linktoc=all, 
	linkcolor=blue} 

\numberwithin{equation}{section}

\theoremstyle{plain}
\newtheorem{theorem}{Theorem}[section]
\newtheorem*{theorem*}{Theorem}

\newtheorem{corollary}[theorem]{Corollary}
\newtheorem{lemma}[theorem]{Lemma}

\newtheorem{innercustomgeneric}{\customgenericname}
\providecommand{\customgenericname}{}
\newcommand{\newcustomtheorem}[2]{%
	\newenvironment{#1}[1]
	{%
		\renewcommand\customgenericname{#2}%
		\renewcommand\theinnercustomgeneric{##1}%
		\innercustomgeneric
	}
	{\endinnercustomgeneric}
}
\newcustomtheorem{ctheorem}{Theorem}
\newcustomtheorem{clemma}{Lemma}

\theoremstyle{definition}

\newtheorem*{example*}{Example}
\newtheorem*{examples*}{Examples}
\newtheorem{remark}[theorem]{Remark}
\newtheorem*{remark*}{Remark}
\newtheorem*{remarks*}{Remarks}
\newtheorem*{note*}{Note}

\newtheoremstyle{named}{}{}{\itshape}{}{\bfseries}{.}{.5em}{#1\thmnote{ #3}}
\theoremstyle{named}

\makeatletter
\patchcmd{\subsection}{\bfseries}{\bfseries\boldmath}{}{}
\makeatother

\DeclareMathAlphabet{\mydutchcal}{U}{dutchcal}{m}{n}

\newcommand{\cPD}{\mydutchcal{PD}}
\newcommand{\cPDO}{\mydutchcal{PDO}}

\newcommand{\dis}{\mathsf{div}}
\newcommand{\len}{\mathsf{len}}

\newcommand{\UU}{\mathbf{U}}

\newcommand{\doi}[1]{\href{https://dx.doi.org/#1}{DOI: #1}}

\title[Convolutive sequences,~II]{Convolutive sequences,~II:~Parametrizations}

\author[S. Chern]{Shane Chern}
\address[S. Chern]{Fakult\"at f\"ur Mathematik, Universit\"at Wien, Oskar-Morgenstern-Platz 1, Wien 1090, Austria}
\email{chenxiaohang92@gmail.com, xiaohangc92@univie.ac.at}

\author[D. Eichhorn]{Dennis Eichhorn}
\address[D. Eichhorn]{Department of Mathematics, University of California Irvine, Irvine, CA 92697, USA}
\email{deichhor@math.uci.edu}

\author[S. Fu]{Shishuo Fu}
\address[S. Fu]{College of Mathematics and Statistics \& Center for Discrete Mathematics, Chongqing University, Chongqing 401331, China}
\email{fsshuo@cqu.edu.cn}

\author[J. A. Sellers]{James A. Sellers}
\address[J. A. Sellers]{Department of Mathematics, University of Minnesota Duluth, Duluth, MN 55811, USA}
\email{jsellers@d.umn.edu}

\date{}



\begin{document}
	
\sloppy

\begin{abstract}
	In recent work, the authors defined a sequence $(a_n)_{n\ge 0}$ to be $m$-convolutive exactly if 
	\begin{align*}
		\sum_{n\ge 0} a_{mn} q^n = \left(\sum_{n\ge 0} a_n q^n\right)^m
	\end{align*}
	for a specific positive integer $m$ and provided proofs of the $2$- and $3$-convolutivity of a small number of sequences arising from primitive eta-products. Since the completion of that work, the authors have discovered many new instances of convolutive eta-products. The main focus of this work is to unify all but one of these instances in a parametric way.
    
\end{abstract}

\maketitle

\section{Introduction}
\label{sec:intro}

In the inaugural paper of this series, the authors introduced the concept of \emph{$m$-convolutivity} for a sequence $(a_n)_{n\ge 0}$ if it satisfies the relation~\cite[Definition~1.1]{CEFS1}
\begin{align}\label{def:m-conv}
	\sum_{n\ge 0} a_{mn} q^n = \left(\sum_{n\ge 0} a_n q^n\right)^m
\end{align}
with $m$ a specific positive integer, and the generating series of $a_n$ is also called $m$-convolutive. Meanwhile, the left-hand side of this defining relation is usually represented in terms of the \emph{unitizing operator of degree $m$},
\begin{align*}
	\UU_m\left(\sum_{n\ge 0} a_n q^n\right) := \sum_{n\ge 0} a_{mn} q^n.
\end{align*}

The main objective of \cite{CEFS1} revolves around the case where the numbers $a_n$ are generated by a primitive eta-product. As in \cite[eq.~(1.3)]{CEFS1}, for distinct positive integers $m_j$ sorted in ascending order and integer exponents $\delta_j$, we consider the \emph{eta-product}
\begin{align}\label{eq:eta-product}
	\prod_{j=1}^J (q^{m_j};q^{m_j})^{\delta_j}_{\infty},
\end{align}
where the standard \emph{$q$-Pochhammer symbol} is given by
\begin{align*}
	(a;q)_\infty := \prod_{k\ge 0} (1-aq^k),
\end{align*}
while the shorthand notation
\begin{align*}
	f_i := (q^i;q^i)_\infty
\end{align*}
will be widely utilized throughout. We say the eta-product in \eqref{eq:eta-product} is \emph{primitive} if the moduli $m_1,\ldots,m_J$ have greatest common divisor $1$. In \cite{CEFS1}, we proved the $2$- and $3$-convolutivity for a small list of sequences of numbers acting as the Fourier coefficients of primitive eta-products. All of these sequences can be found in the OEIS database~\cite{OEIS}.

Now a natural question to ask is --- can we understand most of these convolutive relations, if not all, in a unified manner? In this direction, two of the authors~\cite{FuSel} have already made progress in their study of partitions into odd parts with designated summands. In essence, their result can be rephrased in the following \emph{parametric} way, and the equivalence to its original form in \cite[p.~3, Theorem~1.3]{FuSel} (see Theorem~\ref{th:FS}) will be explained in Section~\ref{sec:para-2-comb}.

\begin{theorem}[Fu--Sellers~\cite{FuSel}, equivalent form]\label{th:para-2-conv}
	Let $X$ be independent of $q$. Then
	\begin{align}\label{eq:para-2-conv}
		\UU_2\left(\frac{f_2^2}{f_1^2} (qX,q/X;q^2)_\infty\right) = \left(\frac{f_2^2}{f_1^2} (qX,q/X;q^2)_\infty\right)^2.
	\end{align}
\end{theorem}

This parametric relation is of particular importance because it places four of the five $2$-convolutive eta-products in our first paper within the same setting. By choosing $X$ to be $1$, $-1$, $\omega:=e^{\frac{2 \pi i}{3}}$, $i:=e^{\frac{\pi i}{2}}$, and $\zeta:=e^{\frac{\pi i}{3}}$, respectively, we have
\begin{align*}
	&(q,q;q^2)_\infty = \frac{f_1^2}{f_2^2}, \quad && (-q,-q;q^2)_\infty = \frac{f_2^4}{f_1^2 f_4^2},\\
	&(\omega q, q/\omega;q^2)_\infty = \frac{f_2 f_3}{f_1 f_6}, \quad && (i q, q/i;q^2)_\infty = \frac{f_4^2}{f_2 f_8},\\
	&(\zeta q, q/\zeta;q^2)_\infty = \frac{f_1 f_4 f_6^2}{f_2^2 f_3 f_{12}}.
\end{align*}
Then the eta-products in \cite[Theorems~3.7, 3.6, 3.8, and 3.4]{CEFS1}, listed as the first four series in Table~\ref{tab:list-2-conv}, are, respectively,
\begin{align*}
	&\frac{f_2^2}{f_1^2} (-q,-q;q^2)_\infty = \frac{f_2^6}{f_1^4 f_4^2}, \quad &&
	\frac{f_2^2}{f_1^2} (\omega q, q/\omega;q^2)_\infty = \frac{f_2^3 f_3}{f_1^3 f_6},\\
	&\frac{f_2^2}{f_1^2} (i q, q/i;q^2)_\infty = \frac{f_2 f_4^2}{f_1^2 f_8}, \quad && 
	\frac{f_2^2}{f_1^2} (\zeta q, q/\zeta;q^2)_\infty = \frac{f_4 f_6^2}{f_1 f_3 f_{12}}.
\end{align*}

It should be noted that the search for convolutive sequences in our first paper~\cite{CEFS1} was conducted only on OEIS entries, and this was insufficient to detect all possible $2$-convolutive eta-products for our parametrization purposes. Hereby, we execute a new search, this time directly for primitive eta-products. In doing so, we identify two more $2$-convolutive examples, represented by the series (II.5) and (II.6) in Table~\ref{tab:list-2-conv}.

\begin{table}[ht]
	\renewcommand{\arraystretch}{2.4} 
	\caption{Primitive $2$-convolutive eta-products}\label{tab:list-2-conv}
	\centering
	\smallskip
	\begin{tabular}{|p{2em}>{\centering\arraybackslash}p{0.2\textwidth}|p{2.5em}>{\centering\arraybackslash}p{0.2\textwidth}|} 
		\hline
		\multicolumn{2}{|l}{\makecell[l]{\textbf{Key:}\\[2pt]$^\mathsection$Newly discovered\\[2pt]$^\dagger$Parametric, Thm.~\ref{th:para-2-conv}}} & \multicolumn{2}{l|}{\makecell[l]{~\\[2pt]~\\[2pt]$^\ddagger$Parametric, Thm.~\ref{th:para-2-conv-II}}}\\[5pt]
		\hline
		(II.1)$^{\dagger}$ &  $\dfrac{f_2^6}{f_1^4 f_4^2}$ & (II.2)$^{\dagger,\ddagger}$ & $\dfrac{f_2^3 f_3}{f_1^3 f_6}$ \\
		(II.3)$^{\dagger}$ & $\dfrac{f_2 f_4^2}{f_1^2 f_8}$ & (II.4)$^{\dagger,\ddagger}$ & $\dfrac{f_4 f_6^2}{f_1 f_3 f_{12}}$\\
		(II.5)$^{\mathsection,\ddagger}$ & $\dfrac{f_1f_2f_6}{f_3^3}$ & (II.6)$^{\mathsection,\ddagger}$ & $\dfrac{f_2f_3f_6}{f_1^2 f_9}$  \\
		(II.7) & $\dfrac{f_6 f_{10}}{f_1 f_{15}}$ &&\\ [10pt]
		\hline
	\end{tabular}
\end{table}

Now the key observation is that
\begin{align*}
    \frac{f_2 f_6}{f_1 f_3} (q,q^2,q,q^2;q^3)_\infty &= \frac{f_1f_2f_6}{f_3^3},\\
    \frac{f_2 f_6}{f_1 f_3} (-q,-q^2,-q,-q^2;q^3)_\infty &= \frac{f_2^3 f_3}{f_1^3 f_6},\\
    \frac{f_2 f_6}{f_1 f_3} (\omega q,q^2/\omega,q/\omega,\omega q^2;q^3)_\infty &= \frac{f_2f_3f_6}{f_1^2 f_9},\\
    \frac{f_2 f_6}{f_1 f_3} (iq,q^2/i,q/i,iq^2;q^3)_\infty &= \frac{f_4 f_6^2}{f_1 f_3 f_{12}}.
\end{align*}
These are (II.5), (II.2), (II.6), and (II.4) in Table~\ref{tab:list-2-conv}, respectively. Also, we find that
\begin{align*}
    \frac{f_2 f_6}{f_1 f_3} (\zeta q,q^2/\zeta,q/\zeta,\zeta q^2;q^3)_\infty = \frac{f_6^3 f_9}{f_3^3 f_{18}}.
\end{align*}
However, this relation does not produce any new information because for $l$ and $m$ coprime, if a series $S(q)$ is $m$-convolutive, then so is the series $S(q^l)$, and vice versa. In view of this fact, the primitive counterpart of the above eta-product, obtained by replacing $q^3$ with $q$, is exactly (II.2).

The previous discussions can be summarized as a second parametrization for $2$-convolutive series, with the two newly discovered primitive eta-products (II.5) and (II.6) encoded.

\begin{theorem}\label{th:para-2-conv-II}
	Let $X$ be independent of $q$. Then
	\begin{align}\label{eq:para-2-conv-II}
		\UU_2\left(\frac{f_2 f_6}{f_1 f_3} (qX,q^2/X,q^2X,q/X;q^3)_\infty\right) = \left(\frac{f_2 f_6}{f_1 f_3} (qX,q^2/X,q^2X,q/X;q^3)_\infty\right)^2.
	\end{align}
\end{theorem}

\begin{remark}
    Both parametrizations \eqref{eq:para-2-conv} and \eqref{eq:para-2-conv-II} do not specialize to the last eta-product, (II.7), in Table~\ref{tab:list-2-conv}, making this series the most mysterious. Such an eta-product will be called \emph{sporadic}, as opposed to the \emph{parametric} ones.
\end{remark}

For $3$-convolutive eta-products, one may wonder if a similar parametric generalization exists. However, in \cite{CEFS1} only four $3$-convolutive examples were recorded, making it less likely to find patterns. As such, we begin with a direct search for $3$-convolutive eta-products, which suggests the primitive examples in Table~\ref{tab:list}, with the series (III.1) and (III.2) already proven in \cite[Theorems~4.3 and 4.4]{CEFS1} and the series (III.1') and (III.2') also justified in \cite[Remark~4.1]{CEFS1}, while all others are newly discovered.

\begin{table}[ht]
	\renewcommand{\arraystretch}{2.4} 
    \caption{Primitive $3$-convolutive eta-products}\label{tab:list}
	\centering
	\smallskip
	\begin{tabular}{|p{2em}>{\centering\arraybackslash}p{0.2\textwidth}|p{2.5em}>{\centering\arraybackslash}p{0.2\textwidth}|} 
		\hline
        \multicolumn{4}{|l|}{\makecell[l]{\textbf{Key:}\\[2pt]$^\mathsection$Newly discovered}}\\[5pt]
		\hline
		\multicolumn{2}{|c|}{$F(q)$} & \multicolumn{2}{|c|}{$F(-q)$} \\
		\hline
		(III.1) &  $\dfrac{f_2f_3^2}{f_1^2f_6}$ & (III.1') & $\dfrac{f_1^2f_4^2f_6^5}{f_2^5f_3^2f_{12}^2}$ \\
		(III.2) & $\dfrac{f_2f_6^3}{f_1f_3f_4f_{12}}$ & (III.2') & $\dfrac{f_1 f_3 }{ f_2^2}$\\
		(III.3)$^{\mathsection}$ & $\dfrac{f_2^7 f_6 }{ f_1^4 f_4^4}$ & (III.3')$^{\mathsection}$ & $\dfrac{f_1^4 f_6 }{ f_2^5}$  \\
		(III.4)$^{\mathsection}$ & $\dfrac{f_2^4 f_3 }{ f_1^3 f_4^2}$ & (III.4')$^{\mathsection}$ & $\dfrac{f_1^3 f_4 f_6^3 }{ f_2^5 f_3 f_{12}}$\\
		(III.5)$^{\mathsection}$ & $\dfrac{f_3 f_4^2 }{ f_1 f_2 f_8}$ & (III.5')$^{\mathsection}$ & $\dfrac{f_1 f_4^3 f_6^3 }{ f_2^4 f_3 f_8 f_{12}}$ \\
		(III.6)$^{\mathsection}$ & $\dfrac{f_2^2 f_6 }{ f_1^2 f_8}$ & (III.6')$^{\mathsection}$ & $\dfrac{f_1^2 f_4^2 f_6 }{ f_2^4 f_8}$ \\
		(III.7)$^{\mathsection}$ & $\dfrac{f_5 f_6 }{ f_1 f_{10}}$ & (III.7')$^{\mathsection}$ & $\dfrac{f_1 f_4 f_6 f_{10}^2 }{ f_2^3 f_5 f_{20}}$  \\ [10pt]
		\hline
	\end{tabular}
\end{table}

We remark that Table~\ref{tab:list} should be considered in the following way. Given a sequence $(a_n)_{n\ge 0}$, its \emph{dual} sequence $(a'_n)_{n\ge 0}$ is defined by the relation $a'_n:=(-1)^na_n$. As noted in \cite[Remark~4.1]{CEFS1}, if a sequence is $3$-convolutive, then so is its dual, and vice versa. In this sense, the series listed in Table~\ref{tab:list} can be naturally paired, and once one series in such a pair is proven to be $3$-convolutive, then so is the other. More precisely, the series labeled by a prime can be obtained by replacing $q$ with $-q$ in the corresponding series; such equalities can be shown by the simple $q$-series relation that
\begin{align*}
	(-q;-q)_\infty = \frac{f_2^3}{f_1f_4}.
\end{align*}

In view of the first six pairs in Table~\ref{tab:list}, the following pattern occurs:
\begin{align*}
	&\frac{f_2f_3^2}{f_1^2f_6} = \frac{f_6}{f_2}(\omega q, q/\omega,\omega q, q/\omega;q^2)_\infty, \quad && \frac{f_1^2f_4^2f_6^5}{f_2^5f_3^2f_{12}^2} = \frac{f_6}{f_2}(\zeta q, q/\zeta,\zeta q, q/\zeta;q^2)_\infty,\\
	&\frac{f_2f_6^3}{f_1f_3f_4f_{12}} = \frac{f_6}{f_2}(-q, -q,\zeta q, q/\zeta;q^2)_\infty, \quad && \frac{f_1 f_3 }{ f_2^2} = \frac{f_6}{f_2}(q,q,\omega q, q/\omega;q^2)_\infty,\\
	&\frac{f_2^7 f_6 }{ f_1^4 f_4^4} = \frac{f_6}{f_2}(-q,-q,-q,-q;q^2)_\infty, \quad && \frac{f_1^4 f_6 }{ f_2^5} = \frac{f_6}{f_2}(q,q,q,q;q^2)_\infty,\\
	&\frac{f_2^4 f_3 }{ f_1^3 f_4^2} = \frac{f_6}{f_2}(-q,-q,\omega q, q/\omega;q^2)_\infty, \quad && \frac{f_1^3 f_4 f_6^3 }{ f_2^5 f_3 f_{12}} = \frac{f_6}{f_2}(q,q,\zeta q, q/\zeta;q^2)_\infty,\\
	&\frac{f_3 f_4^2 }{ f_1 f_2 f_8} = \frac{f_6}{f_2}(\omega q, q/\omega,i q, q/i;q^2)_\infty, \quad && \frac{f_1 f_4^3 f_6^3 }{ f_2^4 f_3 f_8 f_{12}} = \frac{f_6}{f_2}(i q, q/i,\zeta q, q/\zeta;q^2)_\infty,\\
	&\frac{f_2^2 f_6 }{ f_1^2 f_8} = \frac{f_6}{f_2}(-q,-q,i q, q/i;q^2)_\infty, \quad && \frac{f_1^2 f_4^2 f_6 }{ f_2^4 f_8} = \frac{f_6}{f_2}(q,q,i q, q/i;q^2)_\infty.
\end{align*}
Moreover, letting $\rho:=e^{\frac{\pi i}{5}}$, we find that the same pattern also works for the last pair:
\begin{align*}
    \frac{f_5 f_6 }{ f_1 f_{10}} = \frac{f_6}{f_2}(\rho^2 q, q/\rho^2,\rho^4 q, q/\rho^4;q^2)_\infty, \ \quad\  \dfrac{f_1 f_4 f_6 f_{10}^2 }{ f_2^3 f_5 f_{20}} = \frac{f_6}{f_2}(\rho q, q/\rho,\rho^3 q, q/\rho^3;q^2)_\infty.
\end{align*}

Hence, we are led to consider the biparametric $3$-convolutive relation in the next theorem, which surprisingly encodes \emph{all} fourteen primitive eta-products in Table~\ref{tab:list}.

\begin{theorem}\label{th:para-3-conv}
	Let $X$ and $Y$ be independent of $q$. Then
	\begin{align}\label{eq:para-3-conv}
		\UU_3\left(\frac{f_6}{f_2} (qX,q/X,qY,q/Y;q^2)_\infty\right) = \left(\frac{f_6}{f_2} (qX,q/X,qY,q/Y;q^2)_\infty\right)^3.
	\end{align}
\end{theorem}

Notably, if we want the series
\begin{align*}
	\frac{f_6}{f_2} (qX,q/X,qY,q/Y;q^2)_\infty
\end{align*}
in Theorem~\ref{th:para-3-conv} to be an eta-product, while $X,Y\in \{\rho,\rho^2,\rho^3,\rho^4\}$, the only possibilities are given by the last pair in Table~\ref{tab:list}. Meanwhile, if we require $X,Y\in \{1,-1,\omega,i,\zeta\}$, then the series specializes to twelve different primitive eta-products, all appearing among the first six pairs in the same table. There are also three nonprimitive cases:
\begin{align*}
	&\frac{f_6}{f_2} (q,q,-q,-q;q^2)_\infty = \frac{f_2f_6}{f_4^2}, \quad && \frac{f_6}{f_2} (\omega q,q/\omega,\zeta q,q/\zeta;q^2)_\infty = \frac{f_4f_6^2}{f_2^2f_{12}},\\
	&\frac{f_6}{f_2} (iq,q/i,iq,q/i;q^2)_\infty = \frac{f_4^4f_6}{f_2^3f_8^2},
\end{align*}
and their primitive counterparts, obtained by replacing $q^2$ with $q$ throughout each eta-product, are still among the aforementioned twelve members. 

\textbf{Outline of the paper.} In Section~\ref{sec:para-2-comb}, we review the combinatorial motivation described in \cite{FuSel} for our parametric $2$-convolutivity in Theorem~\ref{th:para-2-conv} and demonstrate the equivalence between \eqref{eq:para-2-conv} and the combinatorial relation \eqref{eq:2-conv-PDO-dis-comb}. In addition, we connect a uniparametric specialization of the $3$-convolutive relation in Theorem~\ref{th:para-3-conv} to a signed counting for partitions with designated summands in a manner akin to that in \cite{FuSel}. Following this, in Section~\ref{sec:para-2}, we provide an alternative proof of Theorem~\ref{th:para-2-conv} using a theta-dissection technique, while in Sections~\ref{sec:para-2-II} and \ref{sec:para-3}, we establish Theorems~\ref{th:para-2-conv-II} and \ref{th:para-3-conv}, respectively, based on an analogous but more delicate analysis. Finally, we close in Section~\ref{sec:conclusion} with some comments and questions for future study.

\section{Combinatorial motivation}\label{sec:para-2-comb}

Andrews, Lewis, and Lovejoy~\cite{ALL2002} introduced \emph{partitions with designated summands} as partitions such that exactly one part of each size in the partition is marked. For example, $3'+3+2+2'+2+1'$ is a partition of $13$ with designated summands in which we mark the first part of size $3$, the second part of size $2$, and the only part of size $1$.

Denote by $\cPD$ the set of partitions with designated summands, and further by $\cPDO$ the set of partitions with designated summands wherein all parts are odd. An important discovery of Andrews, Lewis, and Lovejoy is the generating function identity~\cite[p.~52, eq.~(1.6)]{ALL2002}
\begin{align*}
	\sum_{\lambda\in \cPDO} q^{|\lambda|} = \frac{f_4 f_6^2}{f_1 f_3 f_{12}},
\end{align*}
where $|\lambda|$ is the \emph{weight} of the partition $\lambda$, namely, the sum of all its parts. In particular, the eta-product on the right-hand side of the above satisfies the $2$-convolutive relation~\cite[p.~63, Theorem~21]{ALL2002}
\begin{align}\label{eq:2-conv-PDO}
	\UU_2\left(\frac{f_4 f_6^2}{f_1 f_3 f_{12}}\right) = \left(\frac{f_4 f_6^2}{f_1 f_3 f_{12}}\right)^2,
\end{align}
thereby yielding the following combinatorial result:
\begin{align}\label{eq:2-conv-PDO-comb}
	\UU_2\left(\sum_{\lambda\in \cPDO} q^{|\lambda|}\right) = \left(\sum_{\lambda\in \cPDO} q^{|\lambda|}\right)^2.
\end{align}

For a partition $\lambda$ with designated summands, let $\dis(\lambda)$ be the \emph{diversity} of $\lambda$, that is, the number of different part sizes in $\lambda$. In \cite[p.~3, Theorem~1.3]{FuSel}, two of the authors provided a uniparametric generalization of \eqref{eq:2-conv-PDO-comb}.

\begin{theorem}[Fu--Sellers~\cite{FuSel}, original form]\label{th:FS}
	Let $x$ be independent of $q$. Then
	\begin{align}\label{eq:2-conv-PDO-dis-comb}
		\UU_2\left(\sum_{\lambda\in \cPDO} x^{\dis(\lambda)} q^{|\lambda|}\right) = \left(\sum_{\lambda\in \cPDO} x^{\dis(\lambda)} q^{|\lambda|}\right)^2.
	\end{align}
\end{theorem}

Now we show why this combinatorial relation is equivalent to our analytic parametrization in \eqref{eq:para-2-conv}. 

\begin{proof}[Proof of the equivalence between \eqref{eq:para-2-conv} and \eqref{eq:2-conv-PDO-dis-comb}]
	By definition,
	\begin{align*}
		\sum_{\lambda\in \cPDO} x^{\dis(\lambda)} q^{|\lambda|} &= \prod_{i\ge 1} \big(1+xq^{2i-1}+2xq^{2(2i-1)}+3xq^{3(2i-1)}+\cdots\big)\\
		&= \frac{f_2^2}{f_1^2} \prod_{i\ge 1} \big(1+(x-2)q^{2i-1}+q^{4i-2}\big).
	\end{align*}
	Making the change of variables
	\begin{align}\label{eq:x-X}
		X = X(x):= \frac{2-x+\sqrt{x^2-4x}}{2},
	\end{align}
	we have
	\begin{align}
		\sum_{\lambda\in \cPDO} x^{\dis(\lambda)} q^{|\lambda|} = \frac{f_2^2}{f_1^2} (qX,q/X;q^2)_\infty.
	\end{align}
	Since $x$ is independent of $q$, so is $X$. The equivalence between \eqref{eq:para-2-conv} and \eqref{eq:2-conv-PDO-dis-comb} is now clear.
\end{proof}

Noting the appearance of the factors $(qX,q/X;q^2)_\infty$ and $(qY,q/Y;q^2)_\infty$ in the $3$-convolutive relation \eqref{eq:para-3-conv}, our next question is whether it has a combinatorial version in analogy with \eqref{eq:2-conv-PDO-dis-comb}. Although it is possible to interpret these two factors in terms of pairs of partitions in $\cPDO$, this explanation is not as neat as we wish. In this sense, we have to sacrifice the degree of freedom by considering a uniparametric specialization.

For the moment, let $\len_{\mathrm{e}}(\lambda)$ and $\dis_{\mathrm{o}}(\lambda)$ denote the total number of even parts and the number of different odd part sizes in $\lambda\in \cPD$, respectively. We have the following \emph{signed} counting for partitions with designated summands:
\begin{align*}
	&\sum_{\lambda\in\cPD}(-1)^{\len_{\mathrm{e}}(\lambda)}x^{\dis_{\mathrm{o}}(\lambda)}q^{|\lambda|} \\
	&\qquad= \prod_{i\ge 1} \big(1+xq^{2i-1}+2xq^{2(2i-1)}+\cdots\big) \big(1 - q^{2i} + 2q^{4i} - \cdots\big)\\
	&\qquad= \frac{f_2^3 f_6}{f_1^2 f_4^2} \prod_{i\ge 1}\big(1+(x-2)q^{2i-1}+q^{4i-2}\big).
\end{align*}
Using the same change of variables as in \eqref{eq:x-X},
\begin{align*}
	X = X(x):= \frac{2-x+\sqrt{x^2-4x}}{2},
\end{align*}
we derive the generating function identity
\begin{align}\label{eq:PD-signed-counting}
	\sum_{\lambda\in\cPD}(-1)^{\len_{\mathrm{e}}(\lambda)}x^{\dis_{\mathrm{o}}(\lambda)}q^{|\lambda|} = \frac{f_2^3 f_6}{f_1^2 f_4^2} (qX,q/X;q^2)_\infty.
\end{align}
Now the key observation is that the right-hand of \eqref{eq:PD-signed-counting} is exactly the $3$-convolutive series in \eqref{eq:para-3-conv} with $Y=-1$. That is,
\begin{align*}
	\frac{f_2^3 f_6}{f_1^2 f_4^2} (qX,q/X;q^2)_\infty = \frac{f_6}{f_2} (qX,q/X,-q,-q;q^2)_\infty.
\end{align*}
Therefore, we arrive at a neat $3$-convolutive analog to \eqref{eq:2-conv-PDO-dis-comb} in light of Theorem~\ref{th:para-3-conv}.

\begin{corollary}
	Let $x$ be independent of $q$. Then
	\begin{align}\label{eq:para-3-conv-PD}
		\UU_3\left(\sum_{\lambda\in\cPD}(-1)^{\len_{\mathrm{e}}(\lambda)}x^{\dis_{\mathrm{o}}(\lambda)}q^{|\lambda|}\right) = \left(\sum_{\lambda\in\cPD}(-1)^{\len_{\mathrm{e}}(\lambda)}x^{\dis_{\mathrm{o}}(\lambda)}q^{|\lambda|}\right)^3.
	\end{align}
\end{corollary}

\section{First $2$-convolutive parametrization, a theta-dissection proof}\label{sec:para-2}

It is notable that, in \cite{FuSel}, the combinatorial relation \eqref{eq:2-conv-PDO-dis-comb} was demonstrated by the theory of Chebyshev polynomials. In this section, we provide an alternative, and in some sense, more intrinsic proof by a theta-dissection technique applied to the series in \eqref{eq:para-2-conv}.

To begin with, we share the following simple yet useful criterion for proving $m$-convolutive relations.

\begin{lemma}\label{criterion-new}
	Fix $m\geq 2$. Assume $A(q),B(q)\in \mathbb{Z}[[q]]$ with $B(q)$ not identical to zero. Then the product $A(q)B(q^m)$ is $m$-convolutive if and only if 
	\begin{align}\label{eq:criterion-new}
		\UU_m\big(A(q)\big) = A(q)^m\cdot \frac{B(q^m)^m}{B(q)}.    
	\end{align}
\end{lemma}

\begin{proof}
	According to the defining relation \eqref{def:m-conv}, we see that $A(q)B(q^m)$ is $m$-convolutive if and only if 
	\begin{align*}
		\big(A(q)B(q^m)\big)^m &= \UU_m\big(A(q)B(q^m)\big) = \UU_m\big(A(q)\big)\cdot B(q),
	\end{align*}
	which matches \eqref{eq:criterion-new} after rearranging the factors.
\end{proof}

Now we are ready to show Theorem~\ref{th:para-2-conv}.

\begin{proof}[Proof of Theorem~\ref{th:para-2-conv}]
	In Lemma~\ref{criterion-new}, we fix $m=2$ and choose
	\begin{align*}
		A(q) = \frac{1}{f_1^2} (qX,q/X,q^2;q^2)_\infty, \qquad\qquad B(q) = f_1.
	\end{align*}
	Then it remains to show
	\begin{align*}
		\UU_2\big(A(q)\big) = \frac{f_2^2}{f_1^5} (qX,q/X,q^2;q^2)_\infty^2.
	\end{align*}
	Recall from \cite[p.~14, eq.~(1.9.4)]{Hir} with $q$ replaced by $-q$ that
	\begin{align*}
		 \frac{1}{f_1^2}=\frac{f_8^5}{f_2^5f_{16}^2}+2q\frac{f_4^2f_{16}^2}{f_2^5f_8}.
	\end{align*}
	Meanwhile, by the Jacobi triple product identity \cite[p.~1, eq.~(1.1.1)]{Hir}, we have
	\begin{align*}
		(qX,q/X,q^2;q^2)_\infty &= \sum_{m=-\infty}^\infty (-1)^m X^m q^{m^2}\\
		&= \sum_{m=-\infty}^\infty X^{2m} q^{(2m)^2} - \sum_{m=-\infty}^\infty X^{2m+1} q^{(2m+1)^2}\\
		&= (-q^4X^2, -q^4/X^2, q^{8};q^{8})_\infty - qX(-q^{8}X^2,-1/X^2,q^{8};q^{8})_\infty.
	\end{align*}
	Thus,
	\begin{align*}
		\UU_2\big(A(q)\big) &= \frac{f_2^2}{f_1^5} \Bigg(\frac{f_4^5}{f_2^2f_8^2}(-q^2X^2,-q^2/X^2,q^4;q^4)_\infty \\
		&\quad-2qX \frac{f_8^2}{f_4}(-q^4X^2,-1/X^2,q^4;q^4)_\infty\Bigg),
	\end{align*}
	so that it is sufficient to show that
	\begin{align*}
		(qX,q/X,q^2;q^2)_\infty^2 &= \frac{f_4^5}{f_2^2f_8^2}(-q^2X^2,-q^2/X^2,q^4;q^4)_\infty\\
		&\quad-2qX \frac{f_8^2}{f_4}(-q^4X^2,-1/X^2,q^4;q^4)_\infty.
	\end{align*}
	Finally, we note that
	\begin{align*}
		(qX,q/X,q^2;q^2)_\infty^2 &= \sum_{n_1,n_2=-\infty}^\infty (-1)^{n_1-n_2} X^{n_1+n_2} q^{n_1^2+n_2^2}\\
		&= \sum_{\substack{m_1,m_2=-\infty\\m_1\equiv m_2 \bmod{2}}}^\infty (-1)^{m_1} X^{m_2} q^{\frac{1}{2}(m_1^2+m_2^2)}\\
		&= \sum_{l_1,l_2=\infty}^\infty X^{2l_2} q^{2(l_1^2+l_2^2)} - qX \sum_{l_1,l_2=\infty}^\infty X^{2l_2} q^{2(l_1^2+l_2^2)+2(l_1+l_2)}\\
		&= (-q^2,-q^2,q^4;q^4)_\infty(-q^2X^2,-q^2/X^2,q^4;q^4)_\infty\\
		&\quad - qX(-q^4,-1,q^4;q^4)_\infty(-q^4X^2,-1/X^2,q^4;q^4)_\infty,
	\end{align*}
	as desired.
\end{proof}

\section{Second $2$-convolutive parametrization}\label{sec:para-2-II}

Using the same theta-dissection technique but with a more intricate analysis, we can provide an analogous proof for Theorem~\ref{th:para-2-conv-II}. Let us start by rewriting the series in \eqref{eq:para-2-conv-II} as
\begin{align*}
	\frac{f_2 f_6}{f_1 f_3} (qX,q^2/X,q^2X,q/X;q^3)_\infty &= \frac{1}{f_1f_3} (qX,q^5/X,q^6;q^6)_\infty(q^5X,q/X,q^6;q^6)_\infty\\
	&\quad\times \frac{f_2}{f_6} (q^2X,q^4/X,q^4X,q^2/X;q^6)_\infty.
\end{align*}
Now in Lemma~\ref{criterion-new}, we fix $m=2$ and choose
\begin{align*}
	A(q) = \frac{1}{f_1f_3} (qX,q^5/X,q^6;q^6)_\infty(q^5X,q/X,q^6;q^6)_\infty
\end{align*}
and
\begin{align*}
	B(q) = \frac{f_1}{f_3} (qX,q^2/X,q^2X,q/X;q^3)_\infty.
\end{align*}
Then it remains to show
\begin{align*}
	\UU_2\big(A(q)\big) = \frac{f_2^2 f_6^2}{f_1^3 f_3^3} (qX,q^2/X,q^3;q^3)_\infty (q^2X,q/X,q^3;q^3)_\infty.
\end{align*}
Recall from \cite[p.~218, eq.~(25.1.7)]{Hir} that
\begin{align*}
	\frac{1}{f_1f_3}=\frac{f_8^2f_{12}^5}{f_2^2f_4f_6^4f_{24}^2}+q\frac{f_4^5f_{24}^2}{f_2^4f_6^2f_8^2f_{12}}.
\end{align*}
Also, by the Jacobi triple product identity,
\begin{align*}
	&(qX,q^5/X,q^6;q^6)_\infty (q^5X,q/X,q^6;q^6)_\infty\\
	&\quad = \sum_{n_1,n_2=-\infty}^\infty (-1)^{n_1+n_2} X^{n_2-n_1} q^{3(n_1^2+n_2^2)+2(n_1+n_2)}\\
	&\quad = \sum_{\substack{m_1,m_2=-\infty\\m_1\equiv m_2 \bmod{2}}}^\infty (-1)^{m_1} X^{m_2} q^{\frac{3}{2}(m_1^2+m_2^2)+2m_1}\\
	&\quad = \sum_{l_1,l_2=-\infty}^\infty X^{2l_2} q^{6l_1^2+6l_2^2+4l_1} - qX \sum_{l_1,l_2=-\infty}^\infty X^{2l_2} q^{6l_1^2+6l_2^2-2l_1+6l_2}\\
	&\quad = (-q^2,-q^{10},q^{12};q^{12})_\infty (-q^6X^2,-q^6/X^2,q^{12};q^{12})_\infty\\
	&\quad\quad - qX (-q^4,-q^8,q^{12};q^{12})_\infty (-q^{12}X^2,-1/X^2,q^{12};q^{12})_\infty\\
	&\quad = \frac{f_4^2f_6f_{24}}{f_2f_8f_{12}}(-q^6X^2,-q^6/X^2,q^{12};q^{12})_\infty - qX \frac{f_8 f_{12}^2}{f_4 f_{24}}(-q^{12}X^2,-1/X^2,q^{12};q^{12})_\infty.
\end{align*}
It follows that
\begin{align*}
	\UU_2\big(A(q)\big) &= \frac{f_2 f_4 f_6^4}{f_1^3 f_3^3 f_{12}}(-q^3X^2,-q^3/X^2,q^{6};q^{6})_\infty\\
	&\quad - qX \frac{f_2^4 f_6 f_{12}}{f_1^4 f_3^2 f_4}(-q^{6}X^2,-1/X^2,q^{6};q^{6})_\infty.
\end{align*}
Finally, we note that
\begin{align*}
	&(qX,q^2/X,q^3;q^3)_\infty (q^2X,q/X,q^3;q^3)_\infty\\
	&\quad = \sum_{n_1,n_2=-\infty}^\infty (-1)^{n_1+n_2} X^{n_2-n_1} q^{\frac{3}{2}(n_1^2+n_2^2)+\frac{1}{2}(n_1+n_2)}\\
	&\quad = \sum_{\substack{m_1,m_2=-\infty\\m_1\equiv m_2 \bmod{2}}}^\infty (-1)^{m_1} X^{m_2} q^{\frac{3}{4}(m_1^2+m_2^2)+\frac{1}{2}m_1}\\
	&\quad = \sum_{l_1,l_2=-\infty}^\infty X^{2l_2} q^{3l_1^2+3l_2^2+l_1} - qX \sum_{l_1,l_2=-\infty}^\infty X^{2l_2} q^{3l_1^2+3l_2^2-2l_1+3l_2}\\
	&\quad = (-q^2,-q^4,q^6;q^6)_\infty (-q^3X^2,-q^3/X^2,q^6;q^6)_\infty\\
	&\quad\quad - qX (-q,-q^5,q^6;q^6)_\infty (-q^6X^2,-1/X^2,q^6;q^6)_\infty\\
	&\quad = \frac{f_4 f_6^2}{f_2 f_{12}}(-q^3X^2,-q^3/X^2,q^6;q^6)_\infty - qX \frac{f_2^2 f_3 f_{12}}{f_1 f_4 f_6} (-q^6X^2,-1/X^2,q^6;q^6)_\infty,
\end{align*}
so that
\begin{align*}
	&\frac{f_2^2 f_6^2}{f_1^3 f_3^3}(qX,q^2/X,q^3;q^3)_\infty (q^2X,q/X,q^3;q^3)_\infty\\
	&\quad = \frac{f_2 f_4 f_6^4}{f_1^3 f_3^3 f_{12}}(-q^3X^2,-q^3/X^2,q^{6};q^{6})_\infty - qX \frac{f_2^4 f_6 f_{12}}{f_1^4 f_3^2 f_4}(-q^{6}X^2,-1/X^2,q^{6};q^{6})_\infty.
\end{align*}
This matches the expression for $\UU_2\big(A(q)\big)$ we obtained earlier, thereby concluding the proof.

\section{$3$-Convolutive parametrization}\label{sec:para-3}
To prove Theorem~\ref{th:para-3-conv}, we fix $m=3$ in Lemma~\ref{criterion-new} and choose
\begin{align*}
	A(q) = \frac{1}{f_2^3} (qX,q/X,q^2;q^2)_\infty (qY,q/Y,q^2;q^2)_\infty, \qquad\qquad B(q) = f_2.
\end{align*}
It remains to show
\begin{align*}
	\UU_3\big(A(q)\big) = \frac{f_6^3}{f_2^{10}} (qX,q/X,q^2;q^2)_\infty^3 (qY,q/Y,q^2;q^2)_\infty^3.
\end{align*}
Recall from \cite[p.~184, eq.~(21.3.7)]{Hir} that
\begin{align*}
	f_1^3=a(q^3)f_3-3qf_9^3,
\end{align*}
where according to \cite[p.~179, eqs.~(21.1.1) and (21.1.2)]{Hir},
\begin{align*}
	a(q) := 1+6\sum_{n=0}^\infty \left(\frac{q^{3n+1}}{1-q^{3n+1}}-\frac{q^{3n+2}}{1-q^{3n+2}}\right).
\end{align*}
Now replacing $q$ with $\omega q$ and $\omega^2 q$ in the above where $\omega:= e^{\frac{2\pi i}{3}}$, and then multiplying the two results, we have
\begin{align*}
	\frac{1}{f_1^3}=a(q^3)^2\frac{f_9^3}{f_3^{10}}+3qa(q^3)\frac{f_9^6}{f_3^{11}}
	+9q^2\frac{f_9^9}{f_3^{12}}.
\end{align*} 
Meanwhile, by the Jacobi triple product identity, this time with a $3$-dissection, we have
\begin{align*}
	&(qX,q/X,q^2;q^2)_\infty\\
	&\quad= \sum_{m=-\infty}^\infty (-1)^m X^m q^{m^2}\\
	&\quad= \sum_{m=-\infty}^\infty (-1)^m X^{3m} q^{(3m)^2} - \sum_{m=-\infty}^\infty (-1)^m (X^{3m+1}+X^{-3m-1}) q^{(3m+1)^2}\\
	&\quad= F_1(X,q^3) - qXF_2(X,q^3) - qX^{-1}F_3(X,q^3),
\end{align*}
where
\begin{align*}
	F_1(X,q) &:= (q^3X^3, q^3/X^3, q^{6};q^{6})_\infty,\\
	F_2(X,q) &:= (q^{5}X^3,q/X^3,q^{6};q^{6})_\infty,\\
	F_3(X,q) &:= (qX^3,q^{5}/X^3,q^{6};q^{6})_\infty.
\end{align*}
Thus,
\begin{align*}
	\UU_3\big(A(q)\big) &= a(q^2)^2F_1(X,q)F_1(Y,q)\frac{f_6^3}{f_2^{10}}\\
	&\quad - 3qa(q^2)F_1(X,q)\big(YF_2(Y,q)+Y^{-1}F_3(Y,q)\big)\frac{f_6^6}{f_2^{11}}\\
	&\quad - 3qa(q^2)F_1(Y,q)\big(XF_2(X,q)+X^{-1}F_3(X,q)\big)\frac{f_6^6}{f_2^{11}}\\
	&\quad + 9q^2 \big(XF_2(X,q)+X^{-1}F_3(X,q)\big)\big(YF_2(Y,q)+Y^{-1}F_3(Y,q)\big)\frac{f_6^9}{f_2^{12}}.
\end{align*}
Finally, we need \cite[p.~180, eq.~(21.2.4)]{Hir}
\begin{align*}
	(qX,q/X,q^2;q^2)_\infty^3 = a(q^2)F_1(X,q) - 3q \frac{f_6^3}{f_2}\big(XF_2(X,q) + X^{-1}F_3(X,q)\big).
\end{align*}
It follows that
\begin{align*}
	&\frac{f_6^3}{f_2^{10}} (qX,q/X,q^2;q^2)_\infty^3 (qY,q/Y,q^2;q^2)_\infty^3\\
	&\qquad = \frac{f_6^3}{f_2^{10}} \left(a(q^2)F_1(X,q) - 3q \frac{f_6^3}{f_2}\big(XF_2(X,q) + X^{-1}F_3(X,q)\big)\right)\\
	&\qquad\quad\times \left(a(q^2)F_1(Y,q) - 3q \frac{f_6^3}{f_2}\big(YF_2(Y,q) + Y^{-1}F_3(Y,q)\big)\right),
\end{align*}
which is exactly the same as the previous expression for $\UU_3\big(A(q)\big)$ after expanding the product. The proof is therefore complete.

\section{Conclusion}\label{sec:conclusion}

We close this work with three sets of comments and questions. 

First, an important problem not yet explored is the possibility of unifying the two $2$-convolutive parametrizations \eqref{eq:para-2-conv} and \eqref{eq:para-2-conv-II} in a biparametric way similar to that for \eqref{eq:para-3-conv}. Ideally, the related series should take the form
\begin{align*}
    P(q)(qX,q/X;q^2)_\infty(qY,q^2/Y,q^2Y,q/Y;q^3)_\infty,
\end{align*}
where the prefactor $P(q)$ is an eta-product. For the moment, specializing $(qX,q/X;q^2)_\infty$ to an eta-product, which may have the following options
\begin{align*}
    \frac{f_1^2}{f_2^2},\quad \frac{f_2^4}{f_1^2 f_4^2},\quad \frac{f_2 f_3}{f_1 f_6},\quad \frac{f_4^2}{f_2 f_8},\quad \frac{f_1 f_4 f_6^2}{f_2^2 f_3 f_{12}},
\end{align*}
then its product with $P(q)$ should be $\frac{f_2 f_6}{f_1 f_3}$, the prefactor in \eqref{eq:para-2-conv-II}. Thus, $P(q)$ is one of
\begin{align*}
    \frac{f_2^3 f_6}{f_1^3 f_3},\quad\frac{f_1 f_4^2 f_6}{f_2^3 f_3},\quad\frac{f_6^2}{f_3^2},\quad\frac{f_2^2 f_6 f_8}{f_1 f_3 f_4^2},\quad\frac{f_2^3 f_{12}}{f_1^2 f_4 f_6}.
\end{align*}
Similarly, we may specialize $(qY,q^2/Y,q^2Y,q/Y;q^3)_\infty$ as
\begin{align*}
    \frac{f_1^2}{f_3^2},\quad \frac{f_2^2 f_3^2}{f_1^2 f_6^2},\quad \frac{f_3^2}{f_1 f_9},\quad \frac{f_4 f_6}{f_2 f_{12}},\quad \frac{f_1 f_6^2 f_9}{f_2 f_3^2 f_{18}},
\end{align*}
and multiplying it by $P(q)$ should give $\frac{f_2^2}{f_1^2}$, the prefactor in \eqref{eq:para-2-conv}. Then the choices of $P(q)$ in this case include
\begin{align*}
    \frac{f_2^2 f_3^2}{f_1^4},\quad \frac{f_6^2}{f_3^2},\quad \frac{f_2^2 f_9}{f_1 f_3^2},\quad \frac{f_2^3 f_{12}}{f_1^2 f_4 f_6},\quad \frac{f_2^3 f_3^2 f_{18}}{f_1^3 f_6^2 f_9}.
\end{align*}
Now the overlapping expressions for $P(q)$ are
\begin{align*}
    \frac{f_6^2}{f_3^2},\quad \frac{f_2^3 f_{12}}{f_1^2 f_4 f_6}.
\end{align*}
However, a direct verification reveals that neither of these options produces the desired $2$-convolutivity.

Second, since all other $2$- and $3$-convolutive primitive eta-products in this work can be parametrized, it is curious to ask if the series (II.7) in Table~\ref{tab:list-2-conv} is \emph{truly} sporadic, or if there is a missing parametrization associated with it. In addition, the parametric series in Theorems~\ref{th:para-2-conv}, \ref{th:para-2-conv-II}, and \ref{th:para-3-conv} reduce to a primitive eta-product only for a small selection of parameters, all identified in Tables~\ref{tab:list-2-conv} and \ref{tab:list}. It is still unclear if one can construct an infinite family of $2$- or $3$-convolutive sequences whose generating functions are primitive eta-products, or on the opposite side, if the list of primitive $2$- and $3$-convolutive eta-products is finite.

Third, in a separate project~\cite{CEFS3}, we construct a bijective proof for the basic combinatorial relation \eqref{eq:2-conv-PDO-comb}. However, when the contribution of the ``$\dis$'' statistic is inserted, the bijectivity for \eqref{eq:2-conv-PDO-dis-comb} remains open. For the $3$-convolutive case, it is an easy exercise, which will be left to the motivated reader, to show that the coefficient sequences represented in the left-hand column of Table~\ref{tab:list} are all nonnegative. This fact suggests that the series in the left-hand column should be the counting functions for certain partition sets, thereby shedding light on potential combinatorial proofs of the $3$-convolutivity. In this direction, we have witnessed an interesting paper by Liu and Tian~\cite{LiuTian} on a combinatorial treatment of the series (III.1). For a more general consideration, the parametric series in \eqref{eq:para-3-conv} with $X$ and $Y$ replaced by $-X$ and $-Y$, namely,
\begin{align*}
	\frac{f_6}{f_2} (-qX,-q/X,-qY,-q/Y;q^2)_\infty,
\end{align*}
belongs to $\mathbb{N}[X,X^{-1},Y,Y^{-1}][[q]]$, so it should be the counting function for certain partition tuples. It would be intriguing to see a bijective proof of \eqref{eq:para-3-conv} along this line. As we have remarked in Section~\ref{sec:para-2-comb}, such a partition tuple interpretation might not be as neat as we would like to see. Then a combinatorial study of the neat relation \eqref{eq:para-3-conv-PD}, which corresponds to the $Y=-1$ case of \eqref{eq:para-3-conv}, is also of great interest.

\subsection*{Acknowledgements}

Shane Chern was supported by the FWF Austrian Science Fund (grant no.~10.55776/F1002). Shishuo Fu was supported by the National Natural Science Foundation of China (grant no.~12171059) and the Fundamental Research Funds for the Central Universities (grant no.~2025CDJ-IAISYB-008). 

\bibliographystyle{amsplain}

\begin{thebibliography}{9}
	
	\bibitem{ALL2002}
	G. E. Andrews, R. P. Lewis, and J. Lovejoy, Partitions with designated summands, \textit{Acta Arith.} \textbf{105} (2002), no. 1, 51--66.
	
	\bibitem{CEFS1}
	S. Chern, D. Eichhorn, S. Fu, and J. A. Sellers, Convolutive sequences,~I:~Through the lens of integer partition functions, \textit{Exp. Math.} (2026), 1--14. \doi{10.1080/10586458.2025.2604777}.

    \bibitem{CEFS3}
    S. Chern, D. Eichhorn, S. Fu, and J. A. Sellers, Convolutive sequences, ~III:~Bijective proofs of the $2$-convolutivity of the PDO function and beyond, in preparation.
	
	\bibitem{FuSel}
	S. Fu and J. A. Sellers, A refined view of a curious identity for partitions into odd parts with designated summands, \textit{Discrete Math.} \textbf{348} (2025), no. 12, Paper No. 114620, 13 pp.
	
	\bibitem{Hir}
	M. D. Hirschhorn, \textit{The power of $q$. A personal journey}, Springer, Cham, 2017.
	
	\bibitem{LiuTian}
	J.-C. Liu and Y. Tian, A bijective proof of a cubic convolution identity for $3$-regular overpartitions, preprint. \doi{10.13140/RG.2.2.23921.24166}.
	
	\bibitem{OEIS}
	N. J. A. Sloane, \textit{On-Line Encyclopedia of Integer Sequences}, \url{https://oeis.org}.
	
\end{thebibliography}

\newpage
	
\end{document}